\documentclass[12pt]{article}
\usepackage[a4paper, tmargin=3.5cm, bmargin=3.0cm, lmargin=2.5cm, rmargin=2.5cm, textheight=24cm, textwidth=16cm]{geometry}
\usepackage{setspace}
\usepackage{amsfonts}
\usepackage{epsfig}
\usepackage{latexsym}
\usepackage{amsthm}
\usepackage{amsmath}
\usepackage{amssymb}
\usepackage{array}
\usepackage{enumerate}
\usepackage{enumitem}
\usepackage{graphicx}
\usepackage{breqn}
\usepackage{thmtools, thm-restate}
\usepackage{xcolor}
\usepackage{titlesec}

\newtheorem{theorem}{Theorem}[section]

\newtheorem{lemma}{Lemma}[section]
\newtheorem{remark}{Remark}[section]

\newlist{notes}{enumerate}{1}
\setlist[notes]{label=Note: ,leftmargin=*}

\DeclareMathOperator{\diam}{diam}
\DeclareMathOperator{\spec}{spec}

\title{ \textbf{When Are Standard Graph Products Isomorphic?}}
\author{Priti Prasanna Mondal$^{1, 3}$, M. Rajesh Kannan $^{2,}$$^{*}$, Fouzul Atik$^{1}$ \\
        \small $^{1}$Department of Mathematics, SRM University-AP, Andhra Pradesh - 522240, India \\
        \small $^{2}$Department of Mathematics, Indian Institute of Technology Hyderabad, Hyderabad 502285, India \\
        \small $^{3}$Department of Mathematics, SGGS Institute of Engineering and Technology, Maharashtra 431606, India \\
        \small Emails: \tt{pritiprasanna1992@gmail.com; rajeshkannan1.m@gmail.com; fouzulatik@gmail.com} \\
        \small $^{*}$Corresponding author: \tt{rajeshkannan1.m@gmail.com} \\
        }

\date{}

\begin{document}
\maketitle{}

\begin{abstract}
 This article investigates the isomorphism problem for graphs derived from the four standard graph products: Cartesian, Kronecker (direct), strong, and lexicographic product. We provide a complete characterization of all simple connected graphs for which their corresponding products are isomorphic. As a  by-product, we identify a novel family of non-distance-regular graphs that possess fewer than d+1 distinct distance eigenvalues, where d represents the diameter of the graph. This result offers a new perspective on Problem 4.3 posed in \cite{atik}, moving beyond the current approaches.
\end{abstract}

{\bf AMS Subject Classification(2010):} 05C76, 05C50.

\textbf{Keywords:} Graph products, Graph isomorphism, Graph spectrum, Distance-regular graphs.

\section{Introduction}

Graph products form an important area of study in graph theory and are frequently used to construct new families of graphs from existing ones. Let ${G}$ and ${H}$ be two graphs with vertex sets $V({G})=\{ x_{1}, x_{2}, \ldots, x_{n}\}$ and $V({H})=\{ y_{1}, y_{2}, \ldots, y_{m}\}$ respectively. In most cases, the graph product of two graphs ${G}$ and ${H}$ is a new graph whose vertex set is $V({G}) \times V({H})$, the Cartesian product of the sets $V({G})$ and $V({H})$. In this article, we focus on four standard graph products, as named in \cite{prod-hand-book}: the Cartesian, Kronecker, Strong, and Lexicographic products. 
\begin{itemize}
    \item The Cartesian product of graphs was studied by Whitehead and Russell (1912) and later formalized by Sabidussi (1960) \cite{sabi-cart}.  The Cartesian product of the graphs ${G}$ and ${H}$, denoted as ${G} \Box {H}$, is a graph in which $(x_{i}, y_{j})\sim(x_{r},y_{s})$ if either ($x_{i} \sim x_{r}$ in ${G}$ and $y_{j} = y_{s}$) or ($x_{i} = x_{r}$ and $y_{j} \sim y_{s}$ in ${H}$).

    \item In 1962, Weichesel \cite{weichesel} introduced the notion of the Kronecker product of graphs. The Kronecker product of the graphs ${G}$ and ${H}$, denoted as ${G}\otimes {H}$, is a graph in which $(x_{i}, y_{j})\sim(x_{r},y_{s})$ if $x_{i} \sim x_{r}$ in ${G}$ and $y_{j} \sim y_{s}$ in ${H}$. 

    \item The Strong product of the graphs ${G}$ and ${H}$, denoted as ${G}\boxtimes {H}$, is a graph in which $(x_{i}, y_{j})\sim(x_{r},y_{s})$ if either ($x_{i} \sim x_{r}$ in ${G}$ and $y_{j} = y_{s}$) or ($x_{i} = x_{r}$ and $y_{j} \sim y_{s}$ in ${H}$) or ($x_{i} \sim x_{r}$ in ${G}$ and $y_{j} \sim y_{s}$ in ${H})$. This notion was introduced by Sabidussi \cite{sabi-cart}.
    
    \item In \cite{harary-lexi}, Harary introduced the lexicographic product of two graphs. The lexicographic product of the graphs ${G}$ and ${H}$, denoted as ${G}\circ {H}$, is a graph in which $(x_{i}, y_{j})\sim(x_{r},y_{s})$ if $x_{i} \sim x_{r}$ in ${G}$ or $x_{i} = x_{r} ~\mbox{and}~ y_{j} \sim y_{s}$ in ${H}$.
\end{itemize} 
 Graph isomorphism is a fascinating and powerful tool with wide-ranging applications in various fields. It involves studying the problem of identifying symmetries in graphs and, more generally, in combinatorial structures. Let ${G}=(V, E)$ be a graph with the vertex set $V$ and edge set $E$. Two graphs ${G}_{1}=(V_{1},E_{1})$ and ${G}_{2}=(V_{2},E_{2})$ are isomorphic if there is a bijection $\varphi$ from $V_{1}$ to $V_{2}$ such that $\{v,w\} \in E_{1} $  if and only if $\{\varphi(v), \varphi(w)\}\in E_{2}.$ 
It is natural to ask the following question regarding standard graph products: \begin{quote}
    For which graphs $G$ and $H$ are the graphs resulting from two different standard products isomorphic? 
\end{quote}In fact, \cite[Exercise 5.1]{prod-hand-book} poses a specific instance of this question:
    Show that the graphs ${C_{2n+1}} \Box {C_{2n+1}}$ and ${C_{2n+1}} \otimes {C_{2n+1}}$ are isomorphic, where $C_{2n+1}$ denotes the odd cycle of length $2n+1$. What changes if we replace odd cycles with even cycles? 
Motivated by such questions, we consider all pairs among the four standard graph products and provide a complete characterization of the graphs $G$ and $H$ for which the resulting graphs are isomorphic.

In \cite{biggs-agt}, Biggs introduced the distance-regular graphs, which are graphs with a lot of combinatorial symmetry. For a vertex $x$ in $V({G})$, let ${G}_{i}(x)$ denote the set of all vertices in ${G}$ that are at a distance $i$ from $x$. A connected graph $G$ is called \emph{distance regular} if it is regular and for any two vertices $x, y \in V({G})$ at distance $i$, there are precisely $c_i$ neighbors of $y$ in ${G}_{i-1}(x)$ and $b_{i}$ neighbors of $y$ in ${G}_{i+1}(x) $.  That is, for the distance regular graph, if we consider any two pairs of vertices as $(x, y)$ and $(x, z) $ such that $d(x, y) = i = d(x, z) $ for $y \neq z,$ then the number $c_{i}^{y} $ for $y$ and the number $c_{i}^{z}$ for $z$ must be equal. That is,  $c_{i}^{y}=c_{i}^{z}$. Similarly, $b_{i}^{y}=b_{i}^{z}$. For more details, refer to \cite{dist-reg-book-bcn, biggs-agt, dist-reg-survey}.  

The distance matrix of a connected graph ${G}$ is denoted by $D({G})$ and defined as $D({G})=(d_{ij})_{n\times n}$, where $d_{ij}$ is the length (number of edges) of the shortest path between vertices $i$ and $j$. In \cite{distance-spec-LAA}, Lin et al. posed the following problem: If $G$ is a connected graph with diameter at least $d$, then $G$ has at least $d+1$ distinct distance eigenvalues. This question was negatively answered by Atik and Panigrahi in \cite{atik}, where they constructed counterexamples. In the same paper, they further asked whether there exist connected graphs other than distance-regular graphs with diameter $d$ and fewer than $d+1$ distinct distance eigenvalues. This question was partially addressed in \cite{ aalipour2016distance, atik2}, where the authors presented examples. 
In this article, using one of the isomorphism characterizations, we construct an infinite family of non-distance-regular graphs of diameter $d$ with fewer than $d+1$ distinct distance eigenvalues (Remark \ref{iso}).

\section{Some preliminary results}
In this section, we recall some known results that will be used in the latter part of the paper. These results pertain to the degrees, connectivity, and diameter, as well as the adjacency and distance spectra of product graphs.

\begin{theorem}\emph{\cite{prod-hand-book, imrich2000product}}\label{degProd}
    For any two graphs ${G}$ and ${H}$, let $x\in V({G})$ and $y\in V({H})$. Then for any vertex $(x, y)$ in the product graphs,  we have the following: 
    \begin{enumerate}
        \item[(a)] $\deg_{{G} \Box {H}}((x,y))=\deg_{{G}}(x)+\deg_{{H}}(y),$
        \item[(b)] $\deg_{{G} \otimes {H}}((x,y))=\deg_{{G}}(x)\cdot\deg_{{H}}(y),$
        \item[(c)] $\deg_{{G} \boxtimes {H}}((x,y))=(\deg_{{G}}(x)+1)\cdot(\deg_{{H}}(y)+1)-1,$
        \item[(d)] $\deg_{{G} \circ {H}}((x,y))=|V({H})|\cdot\deg_{{G}}(x)+\deg_{{H}}(y).$
    \end{enumerate}
     
\end{theorem}

The following results concern the connectedness of product graphs.\begin{lemma} \emph{(\cite{imri}, Lemma 1.2)} \label{connecCert}
    A Cartesian product ${G} \Box {H}$ is connected if and only if both factors are connected.
\end{lemma}

\begin{theorem}\emph{\cite{weichesel}} \label{connecKron}
Let ${G}={G}_{1}\otimes {G}_{2}$ be the Kronecker product of simple connected graphs ${G}_{1}$ and ${G}_{2}$. Then ${G}$ is connected if and only if either ${G}_{1}$ or ${G}_{2}$  contains an odd cycle. If both ${G}_{1}$ and ${G}_{2}$ are bipartite, then ${G}={G}_{1}\otimes {G}_{2}$ has exactly two components.
\end{theorem}

The following results are about the diameter of the product graphs.
\begin{theorem}\emph{\cite{imri}}\label{diamCart}
    For two simple connected graph ${G}$ and ${H}$, the diameter of the Cartesian product ${G} \Box {H}$ is $\diam({G} \Box {H})= \diam({G}) + \diam({H}).$
\end{theorem}

\begin{theorem}\emph{(\cite{hudiameter}, Corollary 3.9)}\label{diamKron}
Let ${C}_{m}$ be an odd cycle and ${H}$ be a connected graph with order $n$ and diameter $r \ge 1.$
\begin{enumerate}
    \item[(a)] If ${H}$ is bipartite, then $\diam({C}_{m}\otimes {H})=\max\{m,r\}$. Thus, $\diam({C}_{m}\otimes {P}_{n})= \max\{m,n-1\}$  and $\diam({C}_{m}\otimes {C}_{n})=\max\{m,\frac{n}{2}\}$ if $n$ is even.
    \item[(b)] If ${H}={C}_{n}$ and $n$ is odd, then 
\begin{eqnarray*}
 \diam({C}_{m} \otimes {C}_{n})=
     \left\{ \begin{array}{cc} 
        m-1 & if~~ m=n; \\
        \max\{n,\frac{m-1}{2}\} & if ~~ m>n; \\
\max\{m,\frac{n-1}{2}\} & if ~~ m<n.
\end{array} \right.
    \end{eqnarray*}
\end{enumerate}
\end{theorem}

Let $G=(V, E)$ be a simple graph with $n$ vertices. The adjacency matrix of $G$, denoted by $A(G)$, is the $n\times n$ matrix defined as follows. The rows and the columns of $A(G)$ are indexed by $V(G)$. If $i\neq j$ then the $(i, j)$ entry of $A(G)$ is $0$ for vertices $i$ and $j$ non-adjacent, and the $(i, j)$ entry is $1$ for $i$ and $j$ adjacent. The $(i,i)$ entry of $A(G)$ is $0$ for $i=1,\cdots, n.$  The set of all eigenvalues of the adjacency matrix of the graph $G$ is denoted by $ \spec({G})$.
The next result is about the adjacency spectrum of the Cartesian product and the Kronecker product of two graphs. 
\begin{theorem}\emph{\cite{bapat}}\label{specCart}
    Let ${G}$ and ${H}$ be graphs with $n$ and $m$ vertices, respectively. Let $\spec ({G}) = \{ \lambda_{1}, \cdots, \lambda_{n} \}$ and $\spec ({H}) = \{\mu_{1}, \cdots, \mu_{m} \}$ be the eigenvalues of the adjacency matrices of ${G}$ and ${H}$, respectively. Then, 
    
    \begin{enumerate}
        \item[(a)]  $\spec ({G}\Box{H}) = \{\lambda_{i}+\mu_{j} : i=1, \cdots, n;~ j=1, \cdots, m\}.$
        \item[(b)]  $\spec({G}\otimes{H}) = \{ \lambda_{i}\mu_{j}: i=1, \cdots, n;~ j=1, \cdots, m\}.$
    \end{enumerate} 
\end{theorem}

Let $G$ be a connected graph. The set of all eigenvalues of the distance matrix of the graph $G$ is denoted by $ \spec_{D}({G})$. The \emph{transmission $Tr(v)$ of a vertex $v$} of $G$ is defined to be the sum of the distances from $v$ to all other vertices in $G$. That is, $ Tr(v) = \sum_{u\in V} d(u,v).$ A connected graph $G$ is said to $s$-transmission regular if $Tr(v) = s$ for every vertex $ v \in  V.$
   The following result is about the distance eigenvalues of the Cartesian product of transmission regular graphs.
\begin{theorem}\emph{\cite{indu2}}\label{Cartesian}
    Let ${G}$ and ${H}$ be  transmission regular graphs on $m$
and $n$ vertices with transmission regularity $k$ and $t$, respectively. Let $ \spec_{D}({G}) = \{ s, \mu_{2}, \mu_{3}, \cdots, \mu_{m}\}$ and $ \spec_{D}({H}) = \{ t, \eta_{2}, \eta_{3}, \cdots, \eta_{n}\}$. Then the distance spectrum of the Cartesian product of ${G}$ and ${H}$ is given by $$ \spec_{D}({G} \Box {H}) = \{ ns+pt, n\mu_{i}, m\eta_{j}, 0 \} $$ where $i = 2, 3, \cdots, m; j = 2, 3, \cdots, n $ and $0$ is with multiplicity $(m-1)(n-1)$.
\end{theorem}

\section{Main Results }


In this section, we present the main results of this paper. Specifically, we establish the following: Let $G$ and $H$ be two simple connected graphs (with at least two vertices in some results).

\begin{itemize}
    \item The Cartesian product and the Kronecker product of $G$ and $H$ are isomorphic if and only if $G$ and $H$ are odd cycles of the same length (Theorem \ref{isomor}). 
    \item The Cartesian product ${G} \Box {H}$ and the strong product ${G} \boxtimes {H}$ are not isomorphic (Theorem \ref{cart-strong}).
    \item  The Kronecker product ${G}\otimes {H}$ is not isomorphic to the Strong product ${G} \boxtimes {H}$ (Theorem \ref{kron-lexi}).
    \item  The strong product ${G}\boxtimes {H}$ is isomorphic to the lexicographic product ${G} \circ {H}$ if and only if ${H}$ is a complete graph (Theorem \ref{lexi-strong}).
\end{itemize}

\subsection*{$\blacksquare$ Cartesian and Kronecker Products of Graphs}
As mentioned in the introduction, the following question appears in \cite[Exercise 5.1]{prod-hand-book}:
\begin{quote}
Show that the graphs ${C_{2n+1}} \Box {C_{2n+1}}$ and ${C_{2n+1}} \otimes {C_{2n+1}}$ are isomorphic, where $C_{2n+1}$ denotes the odd cycle of length $2n+1$. What changes if we replace odd cycles with even cycles?
\end{quote}
In this section, we completely characterize the graphs for which the Cartesian product and the Kronecker product are isomorphic. For the sake of completeness, we also provide a detailed proof of the isomorphism between ${G} \Box {H}$ and ${G} \otimes {H}$ in the case where both ${G}$ and ${H}$ are odd cycles.
\begin{theorem}\label{isomor}
Let ${G}$ and ${H}$ be simple connected graphs. Then, the graphs ${G} \Box {H}$ and ${G}\otimes {H}$ are isomorphic if and only if ${G}$ and ${H}$ are odd cycles of the same length.
\end{theorem}
\begin{proof}
Let $n \geq 3$ be an odd integer, and ${C}_{n}$ be the cycle graph with vertices labeled with the integers $1, \dots, n$ in cyclic order.
First let us prove that the graph ${C}_{n}\otimes {C}_{n}$ is isomorphic to ${C}_{n} \Box {C}_{n}$.  Define the function $f_{n}:V({C}_{n} \Box {C}_{n})\rightarrow V({C}_{n}\otimes {C}_{n})$ as follows:  $$f_{n}(l,m)=\left((l+m-1)\mbox{(mod $n$)}, (1-l+m) \mbox{(mod $n$)}\right),$$ for $(l,m)\in V({C}_{n} \Box {C}_{n})$ and $ ~1\leq l, m \leq n$. This map $f_n$ is a bijection. For,
 \begin{eqnarray*}
&&f_{n}(l,m)=f_{n}(i,j) ~~~ \mbox{for some}~~~ 1\leq l, m, i, j\leq n\\
 &\implies & \left((l+m-1)\mbox{(mod $n$)}, (1-l+m)\mbox{(mod $n$)}\right)=((i+j-1)\mbox{(mod $n$)}, (1-i+j)\mbox{(mod $n$)}) \\
 &\implies & (l-i)\mbox{(mod $n$)} = (j-m) \mbox{(mod $n$)}, ~\mbox{and}~
 -(l-i)\mbox{(mod $n$)} = (j-m)\mbox{(mod $n$)}\\
 &\implies & (l-i)\mbox{(mod $n$)} = -(l-i) \mbox{(mod $n$)}\\
 &\implies & 2l\mbox{(mod $n$)} = 2i \mbox{(mod $n$)}\\
 &\implies & l\mbox{(mod $n$)} = i \mbox{(mod $n$)},[\mbox{as $n$ is odd}] \\
&\implies & l = i, ~\mbox{as}~ 1\leq l, i \leq n.
 \end{eqnarray*}
Also, $-(l-i)\mbox{(mod $n$)} = (j-m) \mbox{(mod $n$)}$ implies $m=j$. 
Therefore,  for $1\leq l, m, i, j\leq n $, we have  $l=i$ and $m=j$. Thus $f_{n}(l,m)=f_{n}(i,j) $.
Hence, $f_{n}$ is an injective map. As the number of vertices in both graphs is the same, the map $f_n$ is a bijection.

Next, we show $f_{n}$ is an isomorphism. Let the vertices  $(l,m)$ and $(i,j)$ be adjacent in the graph ${C}_{n} \Box {C}_{n}$. Then either $l=i$ and $ m \sim j$ in ${C}_{n}$ or $ l \sim i$ in ${C}_{n}$ and $ m=j.$  Let us first assume that $(l,m) \sim (i,j)$ with $l=i $ and $ j\sim m $ in ${C}_{n}.$ Then $ j = (m\pm 1) (\mbox{mod}~n).$
Note that, by definition, $f_{n}(l,m)=\left((l+m-1)\mbox{(mod $n$)}, (1-l+m) \mbox{(mod $n$)}\right)$ and \begin{align*}
f_{n}(i,j)&= \left((i+j-1)\mbox{(mod $n$)}, (1-i+j) \mbox{(mod $n$)}\right)\\
&=\left((l+m-1\pm 1)\mbox{(mod $n$)}, (1-l+m \pm 1) \mbox{(mod $n$)}\right).
\end{align*}
Note that, $$(l+m-1)\mbox{(mod n)} - (l+m-1 \pm 1) \mbox{(mod n)} = \pm 1\mbox{(mod n)} $$ \mbox{and} 
 $$ ~~~(1-l+m)\mbox{(mod n) }- (1-l+m \pm 1) \mbox{(mod n)} = \pm 1 \mbox{(mod n)}.$$
    
That is, the vertices $$ \left((l+m-1)\mbox{(mod n)}, (1-l+m) \mbox{(mod n)}\right)$$ and $$  \left((l+m-1\pm 1)\mbox{(mod n)}, (1-l+m \pm 1) \mbox{(mod n)}\right)$$ are adjacent.  Thus,
the vertices $ f_{n}(l,m)$ and $ f_{n}(i,j)$ are adjacent in $ {C}_{n} \otimes {C}_{n}$.

Proof of the case $(l,m) \sim (i,j) $ with $l \sim i$ in ${C}_{n}$ and $ m = j$ is similar.  Thus, the graphs ${C}_{n} \Box {C}_{n} $ and $ {C}_{n} \otimes {C}_{n}$ are isomorphic, when $n$ is odd and $n \ge 3$.

Conversely, we now show that if ${G}$ and ${H}$ are not odd cycles of the same length, then the graphs ${G} \Box {H}$ and ${G} \otimes {H}$ are not isomorphic. To establish this, we examine the following four cases:

\textbf{Case I:} Suppose that  at least one of the graphs ${G}$ or ${H}$ is not a cycle graph. Let ${G}$ not be a cycle. Let $\deg_{{G}}(1)\leq \deg_{{G}}(2)\leq \cdots \leq \deg_{{G}}(n)$ be the degree sequence of ${G}$, where $n$ is the order of the graph ${G}$. Then either $ \deg_{{G}}(1)$ or $ \deg_{{G}}(n)$ is not equal to $2$. Let  $\deg_{{H}}(1)\leq \deg_{{H}}(2)\leq \cdots \leq \deg_{{H}}(m)$ be the degree sequence of ${H}$, where $m$ is the order of the graph ${H}$. 

Note that,  by Theorem \ref{degProd}, the minimum degree of the graph ${G} \Box {H}$ is $$\delta_{{G} \Box {H}}=\delta_{{G}}+\delta_{{H}}=\deg_{{G}}(1)+\deg_{{H}}(1),$$  and  the minimum degree of the graph ${G} \otimes {H}$ is $$\delta_{{G} \otimes {H}}=\delta_{{G}} \delta_{{H}}=\deg_{{G}}(1)\deg_{{H}}(1).$$ 
If $\deg_{{G}} (1) \neq 2$, then $\deg_{{G}}(1)+\deg_{{H}}(1)\neq \deg_{{G}}(1)\deg_{{H}}(1)$. This implies $\delta_{{G} \Box {H}}\neq \delta_{{G} \otimes {H}}$. 

The maximum degree of the graph ${G} \Box {H}$ is $$\Delta_{{G} \Box {H}}=\Delta_{{G}}+\Delta_{{H}}=\deg_{{G}}(n)+\deg_{{H}}(m),$$ and, by Theorem \ref{degProd}, the maximum degree of the graph ${G} \otimes {H}$ is $$\Delta_{{G} \otimes {H}}=\Delta_{{G}} \Delta_{{H}}=\deg_{{G}}(n)\deg_{{H}}(m).$$ If $\deg_{{G}}(n) \neq 2$, then $\deg_{{G}}(n)+\deg_{{H}}(m)\neq \deg_{{G}}(n)\deg_{{H}}(m)$. This implies $\Delta_{{G} \Box {H}}\neq \Delta_{{G} \otimes {H}}$. 

Thus, if one of the graphs is not a cycle graph, then either the minimum degree or the maximum degree of the graphs  ${G} \Box {H}$ and $ {G} \otimes {H}$ are not equal. Hence,  the graphs ${G} \Box {H}$ and $ {G} \otimes {H}$ are not isomorphic. 

\textbf{Case II:} Let ${G}$ and ${H}$ be both even cycles. Then ${G} \Box {H}$ is a connected graph (by Theorem \ref{connecCert}),  whereas ${G} \otimes {H}$ is a disconnected graph (by Theorem \ref{connecKron}). So the graphs ${G} \Box {H}$ and $ {G} \otimes {H} $ are not isomorphic.

\textbf{Case III:} Let ${G}$ be an even cycle and ${H}$ be an odd cycle. Let ${G}={C}_{2n}$ and ${H}={C}_{2m+1}$. The eigenvalues of the cycle graph ${C}_{p}$ (by Theorem \ref{specCycl}) are $\lambda_{j}= 2\cos(\frac{2\pi}{p}j);$ for $j=0, 1, 2, \cdots, p-1.$  Note that  $$\lambda_{p-j}= 2\cos(\frac{2\pi}{p}(p-j))=2\cos(\frac{2\pi}{p}j)=\lambda_{j}$$ for $j=0, 1, 2, \cdots, p-1.$ So, if $p=2n$ then the distinct eigenvalues of Cycle graph are $\lambda_{j}= 2\cos(\frac{2\pi}{2n}j);$ for $j=0, 1, 2, \cdots, n.$ Since $\cos x$ is the strictly decreasing function on $[0, \pi]$, so $\lambda_{n}^{{G}}= 2\cos(\pi)=-2$  is the smallest eigenvalue of the graph ${G}={C}_{2n}$. Similarly, $\lambda_{m}^{{H}}= 2\cos(\frac{2m\pi}{2m+1})\neq -2$ is the smallest eigenvalues of the graph ${H}={C}_{2m+1}$. 

Let  $$\lambda_{0}^{{G}}=2>\lambda_{1}^{{G}}> \lambda_{2}^{{G}} > \cdots > \lambda_{n}^{{G}}=-2$$ be the distinct adjacency eigenvalues of  ${G}$,  and   $$\lambda_{0}^{{H}}=2>\lambda_{1}^{{H}}> \lambda_{2}^{{H}} > \cdots > \lambda_{m}^{{H}}\neq -2$$ be the distinct adjacency eigenvalues of ${H}$. Then, by Theorem \ref{specCart},  the smallest adjacency eigenvalue of ${G} \Box {H}$ is $-2 + \lambda_{m}^{{H}}$ and that of ${G} \otimes {H}$ is either $-4$ or $ 2\lambda_{m}^{{H}}$. Since $\lambda_{m}^{{H}} \neq -2 $, we have $-2 + \lambda_{m}^{{H}} \notin   \{-4, 2\lambda_{m}^{{H}}\}.$ 
 Thus the graphs ${G} \Box {H}$ and ${G} \otimes {H}$ are not isomorphic.

\textbf{Case IV:} Let ${G}$ and ${H}$ be  odd cycles of different length. Let ${G}={C}_{2n+1}$ and ${H}={C}_{2m+1}$ where $n\neq m$. 
Similar to the Case III, let
 $\lambda_{n}^{{G}}= 2\cos(\frac{2n\pi}{2n+1})$ and $\lambda_{m}^{{H}}= 2\cos(\frac{2m\pi}{2m+1})$ be the smallest eigenvalues of the graphs ${G}={C}_{2n+1}$ and ${H}={C}_{2m+1}$ respectively. Then, for $n\neq m$, we have $\lambda_{n}^{{G}} \neq\lambda_{m}^{{H}}.$
 The smallest adjacency eigenvalue of ${G} \Box {H}$ is $\lambda_{n}^{{G}} + \lambda_{m}^{{H}},$ and that of ${G} \otimes {H}$ is either $2\lambda_{n}^{{G}}$ or $2\lambda_{m}^{{H}}$. Since $\lambda_{n}^{{G}}\neq \lambda_{m}^{{H}}$ for $n\neq m$,  we have $\lambda_{n}^{{G}} + \lambda_{m}^{{H}} \notin \{2\lambda_{n}^{{G}}, 
  2\lambda_{m}^{{H}}\}$. Thus, the graphs ${G} \Box {H}$ and ${G} \otimes {H}$ are not isomorphic. 
\end{proof}
Next, using the previous theorem, we will show that the graph ${C}_n\otimes {C}_n$ is not distance-regular and that the number of distinct distance eigenvalues is less than the diameter plus one. We begin with the following well-known result.
\begin{theorem}\emph{\cite{bapat}}\label{specCycl}
    For $n\ge2,$ the eigenvalues of the cycle graph ${C}_{n}$ are $2\cos(\frac{2\pi k}{n}), ~k=1, \cdots, n.$
\end{theorem}
\begin{theorem}\label{CnKCn}
     Let  
     $\lambda_{1}>\lambda_{2}\ge \lambda_{3}\ge \cdots \ge \lambda_n$ be the distance eigenvalues of the cycle graph ${C}_{n}$,  where $n \ge 3$ is odd. Then the distance eigenvalues of ${C}_{n}\otimes {C}_{n}$ are $ 2n\lambda_{1}, n\lambda_{i}$ with multiplicity $2$, and $0$ with multiplicity $(n-1)^{2},$ where $i=2,3,\cdots, n.$
\end{theorem}
\begin{proof}
    The proof follows from Theorem \ref{isomor} and Theorem \ref{Cartesian}.
\end{proof}

\begin{figure}[!ht]
   \begin {center}
    \includegraphics[width=3in]{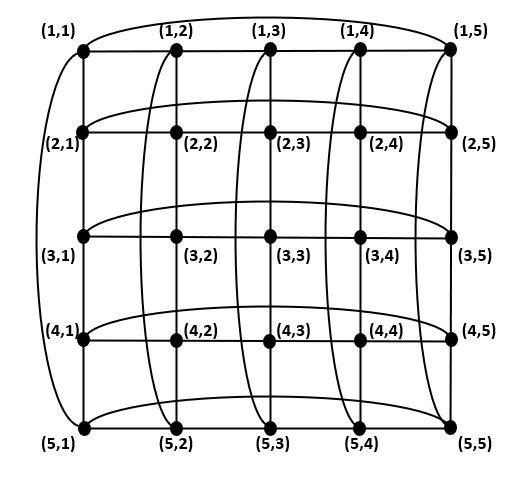}
    \caption{ Graph $ C_5 \Box C_5$}
    \label{cart}
    \end{center}
\end{figure}

 Consider the graph $ {C}_{5} \Box {C}_{5} $ (see Figure \ref{cart}). Let $x=(1,1), y=(1,3), z=(5,5) $. Then $d(x,y)=2 $ and $d(x,z)=2 $.  It is easy to see that $c_2^{y}=1$ and $c_2^{z}=2$. Thus, the graph $ {C}_{5} \Box {C}_{5} $ is not distance-regular. By extending the above argument, we show that the graph $ {C}_{2n+1} \Box {C}_{2n+1}$ is not distance regular. Consider the vertices $x=(1,1), y=(1,3), z=(2n+1, 2n+1) $. Then $c_2^{y}=1$ and $c_2^{z}=2$. Thus  $c_2^{y}\neq c_2^{z}$, and hence $ {C}_{2n+1} \Box {C}_{2n+1} $ is not distance regular graph.

\begin{remark}\label{iso}
    \emph{Let $n$ be an odd integer. The diameter of the graph ${C}_n\otimes {C}_n$ is $n-1$ (by Corollary \ref{diamKron}). The number of distinct distance eigenvalues of ${C}_{n}$ is $\frac{n+1}{2}$, and  the number of distinct distance eigenvalues of ${C}_n\otimes {C}_n$ is $\frac{n+1}{2}+1 = \frac{n+3}{2}$. Since $\frac{n+3}{2}<n$ for any $n\geq 5,$ so the number of distinct distance eigenvalues of ${C}_n\otimes {C}_n$ is less than the diameter plus one. Therefore, when $n$ is odd, ${C}_n\otimes {C}_n$ forms an infinite family of non-distance-regular graphs of diameter $d$ with fewer than $d + 1$ distinct distance eigenvalues. This is another significant observation that contributes to addressing the Problem 4.3 posed in \cite{atik}, which asks: ``Are there connected graphs other than distance-regular graphs with diameter $d$ and having less than $d+1$ distinct distance eigenvalues?"}
\end{remark}


\subsection*{$\blacksquare$ Cartesian and Strong Products of Graphs}
\begin{theorem}\label{cart-strong}
  Let ${G}$ and ${H}$ be simple connected graphs, each with at least two vertices. Then the Cartesian product ${G} \Box {H}$ and the strong product ${G} \boxtimes {H}$ are not isomorphic.
\end{theorem}
\begin{proof}
   Let $ \Delta_{{G}}$ and $\Delta_{{H}} $ denote the maximum vertex degrees of the graphs ${G}$ and ${H}$, respectively. Then, the maximum vertex degree of the graph ${G} \Box {H}$ is $\Delta_{{G}} + \Delta_{{H}}$, and the maximum vertex of the graph ${G} \boxtimes {H}$ is $(\Delta_{{G}} + 1) (\Delta_{{H}}+1) -1$. Suppose that the graphs $ {G} \Box {H} $ and ${G} \boxtimes {H}$ are isomorphic. Then their maximum degrees must be equal:   \[ \Delta_{{G}} + \Delta_{{H}} =(\Delta_{{G}} + 1) (\Delta_{{H}}+1) -1.\] This implies that $\Delta_{{G}}.\Delta_{{H}} =0$. That is, either $\Delta_{{G}} = 0$ or $\Delta_{{H}} = 0$, which is a contradiction. Thus, the graphs $ {G} \Box {H} $ and ${G} \boxtimes {H}$ are not isomorphic.
   \end{proof}

\subsection*{$\blacksquare$ Cartesian and Lexicographic Products of Graphs}
\begin{theorem}
   Let ${G}$ and ${H}$ be connected graphs, each with at least two vertices.  Then the Cartesian product ${G}\Box {H}$ and the Lexicographic product ${G} \circ {H}$ are not isomorphic.
\end{theorem}
\begin{proof}
    The maximum vertex degree of ${G} \Box {H}$ is $\Delta_{{G}} + \Delta_{{H}}$, and the maximum vertex degree of  $ {G} \circ {H}$ is $|V({H})|\Delta_{{G}} + \Delta_{{H}}.$ Suppose that $ {G} \Box {H} $ is isomorphic to ${G} \circ {H}.$ Then   \[  \Delta_{{G}} + \Delta_{{H}} =\vert V(H)\vert \Delta_{{G}} + \Delta_{{H}}.\]
   That is, $\vert V(H) \vert  = 1$, which is a contradiction. So, $ {G} \Box {H} $ is not isomorphic to ${G} \circ {H}$.   
\end{proof}

\subsection*{$\blacksquare$ Kronecker and Strong Products of Graphs}
\begin{theorem}
   Let ${G}$ and ${H}$ be connected graphs, each with at least two vertices.  Then the Kronecker product ${G}\otimes {H}$ is not isomorphic to the Strong product ${G} \boxtimes {H}$.
\end{theorem}
\begin{proof}
    The maximum degree of $ {G} \otimes {H}$ is $\Delta_{G}\cdot \Delta_{H}$, and the maximum degree of $ {G} \boxtimes {H}$ is $(\Delta_{G} + 1) (\Delta_{H}+1) -1.$ Suppose that the graphs $ {G} \otimes {H} $ are ${G} \boxtimes {H}$ isomorphic. Then 
   \[ \Delta_{{G}}. \Delta_{{H}} =(\Delta_{{G}} + 1) (\Delta_{{H}}+1) -1.\]
    That is, $ \Delta_{{G}} +\Delta_{{H}} =0$, 
   which is a contradiction. So, $ {G} \otimes {H} $ is not isomorphic to ${G} \boxtimes {H}$.
   
\end{proof}

\subsection*{$\blacksquare$ Kronecker and Lexicographic Products of Graphs}
\begin{theorem}\label{kron-lexi}
   Let ${G}$ and ${H}$ be connected graphs, each with at least two vertices. Then ${G}\otimes {H}$ is not isomorphic to ${G} \circ {H}$.
\end{theorem}
\begin{proof}
The maximum degree of  ${G} \otimes {H}$ is $\Delta_{G}\cdot \Delta_{H}$, and the maximum degree of ${G} \circ {H}$ is $ \vert V({H})\vert\Delta_{G} +\Delta_{H}.$ Suppose that $ {G} \otimes {H} $ is isomorphic to ${G} \circ {H}$. Then, 
   \[  \Delta_{G}\cdot \Delta_{H} = \vert V({H}) \vert \Delta_{{G}} + \Delta_{{H}}.\] That is, 
    \[ \Delta_{{G}}\cdot(\Delta_{{H}}-\vert V({H}) \vert) =\Delta_{{H}}\]
    Since $(\Delta_{{H}}-\vert V({H}) \vert)<0 $, we have $\Delta_{{G}}\cdot(\Delta_{{H}}-\vert V({H}) \vert)<0$, which implies that $\Delta_{{H}} <0.$ This is a contradiction. Thus, $ {G} \otimes {H} $ is not isomorphic to ${G} \circ {H}$.
   
\end{proof}

\subsection*{$\blacksquare$ Strong and Lexicographic Products of Graphs}
\begin{theorem}\label{lexi-strong}
    Let ${G}$ and ${H}$ be connected graphs, each with at least two vertices. Then, the strong product ${G}\boxtimes {H}$ is isomorphic to the lexicographic product ${G} \circ {H}$ if and only if ${H}$ is a complete graph.
\end{theorem}
 
\begin{proof}
    Let ${H}$ be the complete graph $ K_{m}$ for some $m \in \mathbb{N}$. Define the function $f:V({G} \boxtimes K_{m})\rightarrow V({G}\circ K_{m}) $ as follows:  \[f(i, j)=(i, j) ~\mbox{for} ~(i, j)\in V({G}) \times V(K_{m});~ 1 \leq i \leq n ~\mbox{and}~ 1\leq j \leq m.\]
    
    Now, let us show that $f$ preserves the adjacency relation. Suppose \[(i, j) \sim (r, s)  ~\mbox{in}~ {G} \boxtimes K_{m} ~\mbox{for some}~ 1 \leq i, r \leq n ~\mbox{and}~ 1\leq j, s \leq m. \] 
    
    Then, by the definition of the strong product $\boxtimes$, we have one of the following:  \begin{align*}
     ~ i\sim r~\mbox{in}~{G} ~\mbox{and}~ j\sim s ~\mbox{in}~K_m, ~ \mbox{or,}~ i\sim r~\mbox{in}~{G} ~\mbox{and} ~j= s, ~\mbox{or,}~ i=r~ \mbox{and}~ j\sim s~\mbox{in}~K_m.
\end{align*}
This implies that,
\[   ~ i\sim r ~\mbox{in}~{G}~\mbox{and}~ (j\sim s~\mbox{in}~K_m ~ \mbox{or} ~j= s), ~\mbox{or},~ i=r~ \mbox{and}~ j\sim s ~\mbox{in}~K_m. \]
That is, \[   ~ i\sim r~\mbox{in}~{G} , ~\mbox{or},~ i=r~ \mbox{and}~ j\sim s ~\mbox{in}~K_m. \]

But this is precisely the definition of adjacency in the lexicographic product ${G} \circ K_m$. Hence, $G\boxtimes H$ is isomorphic to ${G} \circ {H},$ if $H$ is a complete graph. 

   Next, let us prove the converse. Let ${H}$ be a graph that is \emph{not} complete. Let $\delta_{{G}}$ and $\delta_{{H}}$ denote the minimum vertex degrees of the graphs ${G}$ and ${H}$, respectively. 
  By Theorem~\ref{degProd}, the minimum degree of the strong product ${G} \boxtimes {H}$ is given by
\[
\delta({G} \boxtimes {H}) = (\delta_{{G}} + 1)(\delta_{{H}} + 1) - 1,
\]
and the minimum degree of the composition (lexicographic product) ${G} \circ {H}$ is
\[
\delta({G} \circ {H}) = |V(H)|\delta_{{G}} + \delta_{H}.
\] Suppose that the graphs $ {G} \boxtimes {H} $ and ${G} \circ {H}$ are isomorphic. Then their minimum degrees must be the same. That is
   \[  (\delta_{{G}}+1)(\delta_{{H}}+1)-1 =\vert V({H}) \vert \delta_{{G}} + \delta_{{H}}. \]
After simplifying, we get    \[  \delta_{{G}}.(\delta_{{H}}+1- m) =0. \]
    Now, since $\delta_{{G}}\neq 0 $ and $ \delta_{{H}}\neq m-1 $, the above gives a contradiction. Thus, $ {G} \boxtimes {H} $ is not isomorphic to ${G} \circ {H}$, when ${H}$ is not complete. 
   
\end{proof}


 \section*{Acknowledgments}
  M. Rajesh Kannan acknowledges financial support from the ANRF-CRG India and  SRC, IIT Hyderabad. 

\section*{Declarations}
\textbf{Conflict of interest:} The authors state that they possess no conflicts of interest.

\bibliographystyle{amsplain}
\bibliography{Ref}

\providecommand{\bysame}{\leavevmode\hbox to3em{\hrulefill}\thinspace}
\providecommand{\MR}{\relax\ifhmode\unskip\space\fi MR }
\providecommand{\MRhref}[2]{%
  \href{http://www.ams.org/mathscinet-getitem?mr=#1}{#2}
}
\providecommand{\href}[2]{#2}
\begin{thebibliography}{10}

\bibitem{aalipour2016distance}
Ghodratollah Aalipour, Aida Abiad, Zhanar Berikkyzy, Jay Cummings, Jessica
  De~Silva, Wei Gao, Kristin Heysse, Leslie Hogben, Franklin~HJ Kenter, Jephian
  C-H Lin, et~al., \emph{On the distance spectra of graphs}, Linear Algebra and
  its Applications \textbf{497} (2016), 66--87.

\bibitem{atik}
Fouzul Atik and Pratima Panigrahi, \emph{On the distance spectrum of distance
  regular graphs}, Linear Algebra and its Applications \textbf{478} (2015),
  256--273.

\bibitem{atik2}
Fouzul Atik and Pratima. Panigrahi, \emph{Families of graphs having few
  distinct distance eigenvalues with arbitrary diameter}, Electron. J. Linear
  Algebra \textbf{29} (2016), 194--205.

\bibitem{bapat}
Ravindra~B Bapat, \emph{Graphs and matrices}, vol.~27, Springer, 2010.

\bibitem{biggs-agt}
Norman Biggs, \emph{Algebraic graph theory}, Cambridge Tracts in Mathematics,
  vol. No. 67, Cambridge University Press, London, 1974. \MR{347649}

\bibitem{dist-reg-book-bcn}
A.~E. Brouwer, A.~M. Cohen, and A.~Neumaier, \emph{Distance-regular graphs},
  Ergebnisse der Mathematik und ihrer Grenzgebiete (3) [Results in Mathematics
  and Related Areas (3)], vol.~18, Springer-Verlag, Berlin, 1989. \MR{1002568}

\bibitem{prod-hand-book}
Richard Hammack, Wilfried Imrich, and Sandi Klav\v~zar, \emph{Handbook of
  product graphs}, second ed., Discrete Mathematics and its Applications (Boca
  Raton), CRC Press, Boca Raton, FL, 2011, With a foreword by Peter Winkler.
  \MR{2817074}

\bibitem{harary-lexi}
Frank Harary, \emph{On the group of the composition of two graphs}, Duke Math.
  J. \textbf{26} (1959), 29--34. \MR{110648}

\bibitem{hudiameter}
Fu-Tao Hu and Jun-Ming Xu, \emph{On the diameter of the kronecker product
  graph}, Mathematical Science Letters \textbf{2} (2013), no.~2, 121--127.

\bibitem{imrich2000product}
Wilfried Imrich and Sandi Klavzar, \emph{Product graphs: Structure and
  recognition}, Wiley, 2000.

\bibitem{imri}
Wilfried Imrich, Sandi Klavzar, and Douglas~F Rall, \emph{Topics in graph
  theory: Graphs and their cartesian product}, CRC Press, 2008.

\bibitem{indu2}
Gopalapillai Indulal, \emph{Distance spectrum of graph compositions.}, Ars
  Math. Contemp. \textbf{2} (2009), no.~1, 93--100.

\bibitem{distance-spec-LAA}
Huiqiu Lin, Yuan Hong, Jianfeng Wang, and Jinlong Shu, \emph{On the distance
  spectrum of graphs}, Linear Algebra Appl. \textbf{439} (2013), no.~6,
  1662--1669. \MR{3073894}

\bibitem{sabi-cart}
Gert Sabidussi, \emph{Graph multiplication}, Math. Z. \textbf{72} (1959/60),
  446--457. \MR{209177}

\bibitem{dist-reg-survey}
Edwin~R. van Dam, Jack~H. Koolen, and Hajime Tanaka, \emph{Distance-regular
  graphs}, Electron. J. Combin. \textbf{DS22} (2016), 156. \MR{4336224}

\bibitem{weichesel}
Paul~M Weichsel, \emph{The kronecker product of graphs}, Proceedings of the
  American Mathematical Society \textbf{13} (1962), no.~1, 47--52.

\end{thebibliography}

\end{document}